\documentclass[12pt,leqno,amscd, amsfonts, amssymb,pstricks,verbatim]{amsart}
\usepackage{amsfonts}

\usepackage{mathrsfs}
\usepackage{amsmath}
\usepackage{amssymb}
\usepackage{hyperref}

\newtheorem{thm}{Theorem}[section]
\newtheorem{lem}[thm]{Lemma}
\newtheorem{cor}[thm]{Corollary}
\newtheorem{prop}[thm]{Proposition}

 \theoremstyle{definition}

\newtheorem{rem}[thm]{Remark}

\numberwithin{equation}{thm}

\def\N{\mathscr N}

\def\CC{\mathscr C}
\def\ggg{\mathfrak{g}}
\def\sss{\mathfrak{s}}
\def\lll{\mathfrak{l}}

\def\Aut {{\rm Aut\,}}

\def\Lie{{\rm Lie}}

\begin{document}

\title[Nilpotent commuting varieties of the Witt algebra]{Nilpotent commuting varieties of the Witt algebra}

\author{Yu-Feng Yao and Hao Chang}

\address{Department of Mathematics, Shanghai Maritime University,
 Shanghai, 201306, China.}\email{yfyao@shmtu.edu.cn}

\address{Department of Mathematics, East China Normal University,
 Shanghai, 200241, China.} \email{hchang@ecnu.cn}



\thanks{This work is supported by the National Natural Science Foundation of China (Grant Nos. 11201293 and 11271130),
the Innovation Program of Shanghai Municipal Education Commission (Grant Nos. 13YZ077 and 12ZZ038), and the Fund of ECNU and SMU for Overseas Studies.}

\begin{abstract}
Let $\ggg$ be the $p$-dimensional Witt algebra over an algebraically closed field $k$ of characteristic $p>3$.
Let $\N=\{x\in\ggg\mid x^{[p]}=0\}$ be the nilpotent variety of $\ggg$, and $\CC(\N):=\{(x,y)\in \N\times\N\mid [x,y]=0\}$ the nilpotent commuting variety of $\ggg$. As an analogue of Premet's result in the case of classical Lie algebras [A. Premet, {\em Nilpotent commuting varieties of reductive Lie algebras}. Invent. Math., 154, 653-683, 2003.], we show that the variety $\CC(\N)$ is reducible and equidimensional. Irreducible components of $\CC(\N)$ and their dimension are precisely given. Furthermore, the nilpotent commuting varieties of Borel subalgebras are also determined.
\end{abstract}

\maketitle

\section{Introduction}
Let $k$ be an algebraically closed field of characteristic $p>0$. For a restricted Lie algebra $\ggg$ over $k$, let $\N=\{x\in\ggg\mid x^{[p]^s}=0\,\,\text{for}\,\,s\gg 0\}$ be the nilpotent variety of $\ggg$. The nilpotent commuting variety $\CC(\N)$ of $\ggg$ is defined as the collection of all $2$-tuples of pairwise commuting elements in $\N$. It is a closed subvariety of $\N\times\N$. For $\ggg=\Lie(G)$ where $G$ is a connected reductive algebraic group and $p$ is good for $G$, Premet \cite{Pr-2} showed that the nilpotent commuting variety $\CC(\N)$ is equidimensional, and the irreducible components are in correspondence with the distinguished nilpotent $G$-orbits in $\N$. The nilpotent commuting variety plays an important role for the study of support varieties of modules over reduced enveloping algebras of $\ggg$ and cohomology theory of the second Frobenius kernel $G_2$ of $G$. Premet's theorem was also proved in characteristic zero. Quite recently, Goodwin and R\"ohrle \cite{GR} gave an analogue of Premet's theorem on the nilpotent commuting varieties of Borel subalgebras of $\ggg$ in the case of characteristic zero. In this paper, we initiate the study of nilpotent commuting varieties of Lie algebras of Cartan type over $k$.

Let $\ggg=W_1$ be the Witt algebra which was found by E. Witt as the first example of non-classical simple Lie algebra in 1930s. As is known to all, $\ggg$ is a restricted Lie algebra, and has a natural $\mathbb{Z}$-grading $\ggg=\sum_{i=-1}^{p-2}\ggg_{[i]}$. Associated with this grading, one has
a filtration $(\ggg_i)_{i\geq -1}$ with $\ggg_i=\sum_{j\geq i}\ggg_{[j]}$ for $i\geq -1$. Let ${\N}=\{x\in\ggg\mid x^{[p]}=0\}$ be the nilpotent
variety of $\ggg$, which is a closed subvariety in $\ggg$. Set $\CC(\N)=\{(x,y)\in\N\times\N\mid [x,y]=0\}$, the nilpotent commuting variety of $\ggg$. It is showed that the variety $\CC(\N)$ is reducible and equidimensional. There are $\frac{p-1}{2}$ irreducible components of the same
dimension $p$ (see Theorem \ref{thm-1}). Consequently, the variety $\CC(\N)$ is not normal (see Corollary \ref{cor-1}). Furthermore, let
${\mathscr{B}}^+=\ggg_0$ be the standard Borel subalgebra of $\ggg$, and $\N({\mathscr{B}}^+)=\{x\in {\mathscr{B}}^+\mid x^{[p]}=0\}=\ggg_1$ the nilpotent
variety of ${\mathscr{B}}^+$. Set $\CC\big({\mathscr{N}}({\mathscr{B}}^+)\big)=\{(x,y)\in {\mathscr{N}}({\mathscr{B}}^+)\times {\mathscr{N}}({\mathscr{B}}^+)\mid [x,y]=0\}$, the nilpotent commuting variety of the Borel subalgebra ${\mathscr{B}}^+$. The variety  $\CC\big({\mathscr{N}}({\mathscr{B}}^+)\big)$ is showed to be reducible and equidimensional. There are $\frac{p-3}{2}$ irreducible components of the same dimension $p$ (see Theorem \ref{thm-2}). Moreover, the variety $\CC\big({\mathscr{N}}({\mathscr{B}}^+)\big)$ is not normal (see Corollary \ref{cor-2}).  As a motivation for further study, it should be mentioned that the nilpotent commuting variety  $\CC\big({\mathscr{N}}({\mathscr{B}}^+)\big)$ of the Borel subalgebra ${\mathscr{B}}^+$ plays a very important role in the cohomology theory of the second Frobenius kernel $G_2$ of $G$, where $G$ is the automorphism group of $\ggg$. To be more precise, it was proved in \cite{SFB} that $\CC\big({\mathscr{N}}({\mathscr{B}}^+)\big)$ is homeomorphic to the spectrum of maximal ideals of the Yoneda algebra $\bigoplus_{i\geq 0} H^{2i}(G_2,k)$ of the second Frobenius kernel $G_2$ of $G$ whenever $p$ is sufficiently large.

\section{Preliminaries}
Throughout this paper, we assume that the ground field $k$ is algebraically closed, and of characteristic $p>3$. Let
${\mathfrak{A}}=k[X]/(X^p)$ be the truncated polynomial algebra of one indeterminate, where $(X^p)$ denotes the ideal of $k[X]$
generated by $X^p$. For brevity, we also denote by $X$ the coset of $X$ in $\mathfrak{A}$. There is a canonical basis $\{1,X,\cdots, X^{p-1}\}$ in
$\mathfrak{A}$. Let $D$ be the linear operator on $\mathfrak{A}$ subject to the rule $DX^i=iX^{i-1}$ for $0\leq i\leq p-1$. Denote by $W_1$ the derivation
algebra of $\mathfrak{A}$, namely the Witt algebra. In the following, we always assume $\ggg=W_1$ unless otherwise stated. By \cite[\S\,4.2]{SF},
$\ggg=\text{span}_k\{X^iD\mid 0\leq i\leq p-1\}$. There is a natural $\mathbb{Z}$-grading on $\ggg$, i.e., $\ggg=\sum_{i=-1}^{p-2}\ggg_{[i]}$,
where $\ggg_{[i]}=kX^{i+1}D,\,-1\leq i\leq p-2$. Associated with this grading, one has the following natural filtration:
$$\ggg=\ggg_{-1}\supset\ggg_0\supset\cdots\supset\ggg_{p-2}\supset 0,$$
where $$\ggg_i=\sum\limits_{j\geq i}\ggg_{[j]}, \,-1\leq i\leq p-2.$$
This filtration is preserved under the action of the automorphism group $G$ of $\ggg$ (cf. \cite{Ch, Re, Wi}). Furthermore, $\ggg$ is a restricted
Lie algebra with the $[p]$-mapping defined as the $p$-th power as usual derivations. Precisely speaking,
$$(X^iD)^{[p]}=
\begin{cases}
0, &\text{if}\,\,i\neq 1, \cr XD, &\text{if}\,\, i=1.
\end{cases}$$
We need the following result on the automorphism group of $\ggg$.

\begin{lem}(cf. \cite{Ch, LN}, see also \cite[Theorem 12.8]{Re})\label{lem-1}
Let $\ggg=W_1$ be the Witt algebra over $k$ and $G=\Aut(\ggg)$. Then the following statements hold.
\begin{itemize}
\item[(i)] $G$ is a connected algebraic group of dimension $p-1$.
\item[(ii)] $\Lie(G)=\ggg_0$.
\end{itemize}
\end{lem}

\begin{rem}
Lemma \ref{lem-1} is not valid for $p=3$. In fact, when $p=3$, the Witt algebra $W_1\cong\sss\lll_2$, and $\Aut(\sss\lll_2)$
has dimension $3$.
\end{rem}

Based on \cite[Proposition 3.3 and Proposition 3.4]{YS}, we get the following useful result by a direct computation.

\begin{lem}\label{lem-2}
Let $\ggg=W_1$ be the Witt algebra. For $x\in\ggg$, let ${\mathfrak{z}}_{\ggg}(x)=\{y\in\ggg\mid [x,y]=0\}$ be the
centralizer of $x$ in $\ggg$. Then
$${\mathfrak{z}}_{\ggg}(x)=
\begin{cases}
kx, &\text{if}\,\,x\in G\cdot D, \cr kx\oplus\ggg_{p-1-i}, &\text{if}\,\,  x\in\ggg_i\setminus\ggg_{i+1},\,1\leq i<\frac{p-1}{2},\cr
\ggg_{p-1-i}, &\text{if}\,\, x\in\ggg_i\setminus\ggg_{i+1},\,i\geq\frac{p-1}{2}.
\end{cases}$$
\end{lem}

\begin{rem}
For $x\in\ggg_1$, let ${\mathfrak{z}}_{\ggg_1}(x)=\{y\in\ggg_1\mid [x,y]=0\}$ be the centralizer of $x$ in $\ggg_1$, then
${\mathfrak{z}}_{\ggg_1}(x)={\mathfrak{z}}_{\ggg}(x)$.
\end{rem}

\section{Nilpotent commuting variety of the Witt algebra}

Keep in mind that $\ggg=W_1$ is the Witt algebra over $k$. Set $\N=\{x\in\ggg\mid x^{[p]}=0\}$, which is
a closed subvariety of $\ggg$. Then $\N$ is just the set of all nilpotent elements in $\ggg$. In the
literature, $\N$ is usually called the nilpotent cone or nilpotent variety of $\ggg$. The variety $\N$ was extensively
studied by Premet in \cite{Pr-1}. The following result is due to Premet.

\begin{lem}(cf. \cite[Theorem 2 and Lemma 4]{Pr-1} or \cite[Lemma 3.1]{YS})\label{lem-3}
Keep notations as above, then the following statements hold.
\begin{itemize}
\item[(i)] The orbit $G\cdot D$ is open and dense in $\N$. Moreover, it coincides with $(\ggg\setminus\ggg_{0})\cap \N$.
\item[(ii)] We have decomposition $\N=G\cdot D\cup\ggg_1$.
\item[(iii)] $\dim\N=p-1$.
\end{itemize}
\end{lem}

Let $\CC(\N):=\{(x,y)\in\N\times\N\mid [x,y]=0\}$, the nilpotent commuting variety of $\ggg$. Obviously, the Zariski closed set
$\CC(\N)$ is preserved by the diagonal action of $G$ on $\N\times\N$. In this section, we study the structure of the variety
$\CC(\N)$.

For $i\in\{1,\cdots, p-2\}$, set
$$C(i):=\{(x,y)\in\N\times\N\mid x\in\ggg_i\setminus\ggg_{i+1},\, [x,y]=0\}.$$
Let
$$C(0)=\{(x,ax)\mid x\in\N,\,a\in k\}$$
and
$$C(p-1)=\{(0,x)\mid x\in\N\}.$$
It is obvious that $C(p-1)$ is a closed subvariety of dimension $p-1$. Set
$${\mathfrak{C}}(i)=\overline{C(i)}\,\,\text{for}\,\,0\leq i\leq p-1.$$
We have the following preliminary result describing the nilpotent commuting variety $\CC(\N)$ of $\ggg$, the proof of which is straightforward.

\begin{lem}\label{lem-4}
Let $\ggg$ be the Witt algebra, $\N$ the nilpotent variety. Then
$\CC(\N)=\bigcup\limits_{i=0}^{p-1} C(i)$. Henceforth,
$\CC(\N)=\bigcup\limits_{i=0}^{p-1} {\mathfrak{C}}(i)$.
\end{lem}

\begin{lem}\label{lem-5}
${\mathfrak{C}}(i)$ is irreducible for any $0\leq i\leq p-1$, and
$$\dim {\mathfrak{C}}(i)=
\begin{cases}
p, &\text{if}\,\,\,0\leq i<\frac{p-1}{2}, \cr p-1, &\text{if}\,\,\, \frac{p-1}{2}\leq i\leq p-1.
\end{cases}$$
Moreover, ${\mathfrak{C}}(p-1)\subseteq {\mathfrak{C}}(0)$.
\end{lem}

\begin{proof}
Obviously, $\mathfrak{C}(0)$ and ${\mathfrak{C}}(p-1)$ are irreducible varieties of dimension $p$ and $p-1$, respectively. For $1\leq i<\frac{p-1}{2}$,
let
$$
\aligned
\varphi: \,\, (\ggg_i\setminus\ggg_{i+1})\times\ggg_{p-1-i}\times {\mathbb{A}}^1
&\longrightarrow C(i)\cr  (x, z, a) & \longmapsto
(x, ax+z)
\endaligned
$$
be the canonical morphism. It follows from Lemma \ref{lem-2} that $\varphi$ is bijective, so that ${\mathfrak{C}}(i)$ is irreducible, and
$$\dim {\mathfrak{C}}(i)=\dim (\ggg_i\setminus\ggg_{i+1})+\dim\ggg_{p-1-i}+1=(p-1-i)+i+1=p.$$
For $\frac{p-1}{2}\leq i\leq p-2$, let
$$
\aligned
\psi: \,\, (\ggg_i\setminus\ggg_{i+1})\times\ggg_{p-1-i}
&\longrightarrow C(i)\cr  (x, y) & \longmapsto
(x, y)
\endaligned
$$
be the canonical morphism. It follows from Lemma \ref{lem-2} that $\psi$ is an isomorphism, so that ${\mathfrak{C}}(i)$ is irreducible, and
$$\dim {\mathfrak{C}}(i)=\dim (\ggg_i\setminus\ggg_{i+1})+\dim\ggg_{p-1-i}=(p-1-i)+i=p-1.$$

Fix $x\in\N$, then
$$\{(\lambda x, x)\mid \lambda\in k^{\times}\}\subseteq C(0).$$
Since
$$\{(\lambda x, x)\mid \lambda\in k^{\times}\}\cong k^{\times},$$
it follows that
$$\{(ax, x)\mid a\in k\}=\overline{\{(\lambda x, x)\mid \lambda\in k^{\times}\}}\subseteq \overline{C(0)}={\mathfrak{C}}(0).$$
In particular, $(0,x)\in {\mathfrak{C}}(0)$ for any $x\in\N$, i.e.,
$${\mathfrak{C}}(p-1)=\{(0,x)\mid x\in\N\}\subseteq {\mathfrak{C}}(0).$$
\end{proof}

As a direct consequence, we have

\begin{cor}
Let $\ggg=W_1$ be the Witt algebra, $\N$ the nilpotent variety of $\ggg$, and $\CC(\N)$ the nilpotent commuting variety of $\ggg$.
Then $\dim\CC(\N)=p$.
\end{cor}

Combining Lemma \ref{lem-4} with Lemma \ref{lem-5}, we get the following result which determines the possible irreducible components of the nilpotent commuting variety $\CC(\N)$.
\begin{prop}\label{prop-1}
Let $\ggg=W_1$ be the Witt algebra, $\N$ the nilpotent variety of $\ggg$. Let $\CC(\N)$ be the nilpotent commuting
variety of $\ggg$. Then each irreducible component of $\CC(\N)$ is of the form ${\mathfrak{C}}(i)$ for some $i\in\{0,1,\cdots, p-2\}$.
\end{prop}

Now we are ready for the main result of this section.

\begin{thm}\label{thm-1}
Let $\ggg=W_1$ be the Witt algebra, $\N$ the nilpotent variety of $\ggg$. Then the nilpotent commuting variety $\CC(\N)$ of $\ggg$ is reducible and
equidimensional. More precisely, $\CC(\N)=\bigcup\limits_{i=0}^{(p-3)/2}{\mathfrak{C}}(i)$ is the decomposition of $\CC(\N)$ into irreducible components.
\end{thm}

\begin{proof}
We divide the proof into several steps.

\textbf{Step 1:}  The group $GL(2,k)$ acts on $\ggg\times\ggg$ via
\[
\left(\begin{array}{cc}
\alpha& \beta\\
\gamma&\delta
\end{array}\right)\cdot(x,y)=(\alpha x+\beta y, \gamma x+\delta y).
\]
Since any linear combination of two commuting elements in $\N$ is again in $\N$, the nilpotent commuting variety $\CC(\N)$ is $GL(2,k)$-invariant.
As $GL(2,k)$ is a connected group, it fixes each irreducible component of $\CC(\N)$. In particular, each irreducible component of $\CC(\N)$ is invariant
under the involution $\sigma:\,(x,y)\mapsto (y,x)$ on $\N\times\N$.

\textbf{Step 2:} Let
$$
\aligned
\pi_1: \,\, \N\times\N
&\twoheadrightarrow \N\cr  (x, y) &\mapsto
x
\endaligned
$$
be the canonical projection. Then
$$\pi_1(C(i))=\ggg_i\setminus\ggg_{i+1},\,1\leq i\leq p-2,$$
and $$\pi_1(C(0))=\N,$$
so that
$$\pi_1({\mathfrak{C}}(i))=\pi_1(\overline{C(i)})=\overline{\ggg_i\setminus\ggg_{i+1}}=\ggg_i$$
and
$$\pi_1({\mathfrak{C}}(0))=\pi_1(\overline{C(0)})=\overline{\N}=\N.$$
It follows that ${\mathfrak{C}}(i)\neq {\mathfrak{C}}(j)$ for distinct $i,j\in\{0,\cdots, p-2\}$.

\textbf{Step 3:} If ${\mathfrak{C}}(i)$ is an irreducible component of $\CC(\N)$ for some
$i\geq 1$, we aim to show that $i\leq\frac{p-1}{2}$. For any $x\in\ggg_i\setminus\ggg_{i+1}$ and $y\in {\mathfrak{z}}_{\ggg}(x)$, since
$(x,y)\in {\mathfrak{C}}(i)$, it follows from Step 1 that $(y,x)\in {\mathfrak{C}}(i)$. Consequently,
$$y=\pi_1(y,x)\in \pi_1({\mathfrak{C}}(i))=\ggg_i.$$
Hence, ${\mathfrak{z}}_{\ggg}(x)\subseteq \ggg_i$. It follows from Lemma \ref{lem-2} that $\ggg_{p-1-i}\subseteq
{\mathfrak{z}}_{\ggg}(x)\subseteq \ggg_i$. Hence, $p-1-i\geq i$, i.e., $i\leq \frac{p-1}{2}$.

In conclusion, the set of possible irreducible components in $\CC(\N)$  is $\{{\mathfrak{C}}(i)\mid 0\leq i\leq \frac{p-1}{2}\}$.

\textbf{Step 4:} ${\mathfrak{C}}(i)$ is an irreducible component of $\CC(\N)$ for $0\leq i\leq \frac{p-3}{2}$. Indeed, if ${\mathfrak{C}}(i)$ is not an irreducible component, it must be contained in ${\mathfrak{C}}(j)$ for some $0\leq j\leq\frac{p-1}{2}$ and $j\neq i$ by Step 3. Moreover, we get ${\mathfrak{C}}(i)={\mathfrak{C}}(j)$ by comparing the dimension. This contradicts the assertion in Step 2.

\textbf{Step 5:} By Lemma \ref{lem-2},
$$
\aligned C(\frac{p-1}{2})=&\{(x,y)\mid x\in\ggg_{\frac{p-1}{2}}\setminus\ggg_{\frac{p+1}{2}},\, [x,y]=0\}\cr
=&\{(x,y)\mid x\in\ggg_{\frac{p-1}{2}}\setminus\ggg_{\frac{p+1}{2}},\, y\in\ggg_{\frac{p-1}{2}} \}.
\endaligned $$
It follows that
$${\mathfrak{C}}(\frac{p-1}{2})=\overline{C(\frac{p-1}{2})}=\ggg_{\frac{p-1}{2}}\times\ggg_{\frac{p-1}{2}}.$$
Moreover,
$${\mathfrak{C}}(\frac{p-1}{2})\subseteq\bigcup\limits_{i=0}^{(p-3)/2}{\mathfrak{C}}(i).$$
In fact, for any $(x,y)\in\ggg_{\frac{p-1}{2}}\times\ggg_{\frac{p-1}{2}}$, we claim that $(x,y)\in {\mathfrak{C}}(i)$ for some $i\in\{0,\cdots, \frac{p-3}{2}\}$. We divide the discussion into the following cases.

\textbf{Case 1:} $x=0$ or $y=0$.

In this case, it is obvious that $(x,y)\in {\mathfrak{C}}(0)$.

\textbf{Case 2:} $y\in\ggg_j\setminus\ggg_{j+1}$ for some $j>\frac{p-1}{2}$.

In this case, set $i=p-1-j<\frac{p-1}{2}$, then
$$\{(u,y)\in\N\times\N\mid u\in\ggg_i\setminus\ggg_{i+1}\}\subseteq C(i).$$
It follows from Lemma \ref{lem-2} that
$$(x,y)\in\{(v,y)\in\N\times\N\mid v\in\ggg_i\}=\overline{\{(u,y)\in\N\times\N\mid u\in\ggg_i\setminus\ggg_{i+1}\}}\subseteq {\mathfrak{C}}(i).$$

\textbf{Case 3:} $x\in\ggg_j\setminus\ggg_{j+1}$ for some $j> \frac{p-1}{2}$.

According to Case 2, $(y,x)\in {\mathfrak{C}}(i)$ for some $i<\frac{p-1}{2}$. Since
\[
(x,y)=\left(\begin{array}{cc}
0& 1\\
1&0
\end{array}\right)\cdot (y,x),
\]

it follows from Step 1 and Step 4 that $(x,y)\in {\mathfrak{C}}(i)$.

\textbf{Case 4:} $x,y \in\ggg_{\frac{p-1}{2}}\setminus\ggg_{\frac{p+1}{2}}$.

In this case, $y=ax+z$ for some $a\in k^{\times}$ and $z\in\ggg_j$ with $j> \frac{p-1}{2}$. Since
\[
(x,y)=\left(\begin{array}{cc}
1&0\\
a&1
\end{array}\right)\cdot (x,z),
\]
it follows from Step 1, Step 4, Case 1 and Case 2 that $(x,y)\in {\mathfrak{C}}(i)$ for $i=p-1-j<\frac{p-1}{2}$ or $i=0$.

In conclusion, $(x,y)\in {\mathfrak{C}}(i)$ for some $i\in\{0,\cdots, \frac{p-3}{2}\}$.

\textbf{Step 6:} It follows from Step 4 and Step 5 that the set of irreducible components of $\CC(\N)$ is exactly $\{{\mathfrak{C}}(i)\mid 0\leq i\leq
\frac{p-3}{2}\}$, so that $\CC(\N)=\bigcup\limits_{i=0}^{(p-3)/2}{\mathfrak{C}}(i)$ is the decomposition of $\CC(\N)$ into irreducible components.

The proof is completed.
\end{proof}

Since
$(0,0)\in\bigcap\limits_{i=0}^{(p-3)/2} {\mathfrak{C}}(i),$
the following result is a direct consequence of Theorem \ref{thm-1}.

\begin{cor}\footnote{We thank Nham V. Ngo for his helpful discussion.}\label{cor-1}
Let $\ggg=W_1$ be the Witt algebra, $\N$ the nilpotent variety. Then the nilpotent commuting variety $\CC(\N)$ is not normal.
\end{cor}

\section{Nilpotent commuting varieties of Borel subalgebras in the Witt algebra}
Let $\ggg=W_1$ be the Witt algebra and $\mathscr{B}$ be a Borel subalgebra. Let $\mathscr{N}(\mathscr{B})$ be the nilpotent variety of
$\mathscr{B}$, and
$$\CC({\mathscr{N}\big(\mathscr{B})}\big)=\{(x,y)\in {\mathscr{N}(\mathscr{B})\times \mathscr{N}(\mathscr{B})}\mid [x,y]=0\}$$
the nilpotent commuting variety of $\mathscr{B}$. According to \cite{YC}, $\mathscr{B}$ is conjugate to $\mathscr{B}^+$ or $\mathscr{B}^-$ under
the automorphism group $G=\Aut(\ggg)$ of $\ggg$, where ${\mathscr{B}}^+=\ggg_0$ and ${\mathscr{B}}^-=$span$_k\{D, XD\}$ are the so-called standard
Borel subalgebras. It is easy to check that ${\mathscr{N}(\mathscr{B}^-)}=kD$ and $\CC(\N(\mathscr{B}^-))=\mathscr{N}(\mathscr{B}^-)\times
\mathscr{N}(\mathscr{B}^-)$. In the following, we always assume $\mathscr{B}=\mathscr{B}^+$. In this case, ${\N(\mathscr{B})}=\ggg_1$. We will
determine the structure of the nilpotent commuting variety $\CC(\ggg_1)$ of the Borel subalgebra $\mathscr{B}=\mathscr{B}^+$.

Set
$$C(p)=\{(0,x)\mid x\in\ggg_1\},\,\,{\mathfrak{C}}(p)=\overline{C(p)}.$$
We have the following preliminary result describing the nilpotent commuting variety $\CC(\ggg_1)$ of the Borel subalgebra $\mathscr{B}^+$,
the proof of which is straightforward.
\begin{lem}\label{lem-6}
Let $\ggg$ be the Witt algebra. Then $\CC(\ggg_1)=C(p)\cup\big(\bigcup\limits_{i=1}^{p-2} C(i)\big)$. Henceforth,
$\CC(\ggg_1)={\mathfrak{C}}(p)\cup\big(\bigcup\limits_{i=1}^{p-2} {\mathfrak{C}}(i)\big)$.
\end{lem}

The following result describes the possible irreducible components of $\CC(\ggg_1)$.

\begin{prop}\label{prop-2}
Let $\ggg$ be the Witt algebra. Then each irreducible component of the nilpotent commuting variety $\CC(\ggg_1)$ is of the form ${\mathfrak{C}}(i)$ for some $i\in\{1,\cdots, p-2\}$.
\end{prop}

\begin{proof}
It follows from Lemma \ref{lem-5} and Lemma \ref{lem-6} that each irreducible component of $\CC(\ggg_1)$ is of the form ${\mathfrak{C}}(i)$ for some $i\in\{1,\cdots, p-2, p\}$. We claim that
$${\mathfrak{C}}(p)\subseteq \bigcup\limits_{i=1}^{p-2}{\mathfrak{C}}(i),$$
from which the assertion follows.

Let $x\in\ggg_1$, then either $x=0$ or there exists a unique $i\in\{1,\cdots, p-2\}$ such that $x\in\ggg_i\setminus\ggg_{i+1}$.

\textbf{Case 1:} $x=0$.

In this case, it is obvious that $(0,0)\in {\mathfrak{C}}(j)$ for any $1\leq j\leq p-2$.

\textbf{Case 2:} $x\in\ggg_i\setminus\ggg_{i+1}$.

In this case,
$$(0,x)\in\{(ax, x)\mid a\in k\}=\overline{\{(ax, x)\mid a\in k^{\times}\}}\subseteq \overline{C(i)}={\mathfrak{C}}(i).$$

Therefore,
$${\mathfrak{C}}(p)\subseteq \bigcup\limits_{i=1}^{p-2}{\mathfrak{C}}(i).$$
We are done.
\end{proof}

We are now in the position to present the main result of this section.

\begin{thm}\label{thm-2}
Let $\ggg=W_1$ be the Witt algebra. Then the nilpotent commuting variety $\CC(\ggg_1)$ of the Borel subalgebra $\mathscr{B}^+$ is reducible
and equidimensional. More precisely, $\CC(\ggg_1)=\bigcup\limits_{i=1}^{(p-3)/2}{\mathfrak{C}}(i)$ is the decomposition of $\CC(\ggg_1)$ into
irreducible components. In particular, $\dim \CC(\ggg_1)=p$.
\end{thm}

\begin{proof}
The proof is similar to that of Theorem \ref{thm-1}.
\end{proof}

\begin{rem}
Let $G=\Aut(\ggg)$ be the automorphism group of $\ggg$. Since $\Lie(G)=\ggg_0={\mathscr{B}}^+$, it follows from \cite{SFB} that the nilpotent commuting variety $\CC(\ggg_1)$ of the Borel subalgebra $\mathscr{B}^+$ is homeomorphic to the spectrum of maximal ideals of the Yoneda algebra $\bigoplus_{i\geq 0} H^{2i}(G_2, k)$ of the second Frobenius kernel $G_2$ of $G$ provided that $p$ is sufficiently large.
\end{rem}

Since $(0,0)\in \bigcap\limits_{i=1}^{(p-3)/2}{\mathfrak{C}}(i)$,
the following result is a direct consequence of Theorem \ref{thm-2}.

\begin{cor}\label{cor-2}
Let $\ggg=W_1$ be the Witt algebra. Then the nilpotent commuting variety $\CC(\ggg_1)$ is not normal.
\end{cor}

\subsection*{Acknowledgements}
This work was done during the visit of the authors to the department of mathematics in the University
of Kiel in 2012-2013. The authors would like to express their sincere gratitude to professor Rolf Farnsteiner for his invitation and hospitality. The authors also thank the department of mathematics for providing an excellent atmosphere.

\end{document}